\documentclass[a4paper,12pt]{article}

\usepackage{amsmath, amssymb, amsthm}
\usepackage{hyperref}

\usepackage{enumerate}
\usepackage{bm}
\usepackage{graphicx}

\usepackage{pgf,tikz}
\usetikzlibrary{arrows}

\usepackage{color}

\usepackage{algorithm, algorithmic}

\newcommand{\mS}{\mathcal{S}}
\newcommand{\mP}{\mathcal{P}}
\newcommand{\mL}{\mathcal{L}}
\newcommand{\mI}{\mathcal{I}}

\newcommand{\mN}{\mathcal{N}}
\newcommand{\mH}{\mathcal{H}}
\newcommand{\G}{\Gamma}

\newcommand{\mG}{\mathcal{G}}
\newcommand{\mO}{\mathcal{O}}
\newcommand{\N}{\mathbb{N}}
\newcommand{\dist}{\mathrm{d}}

\newcommand{\HJ}{\mathsf{HJ}}

\mathchardef\mh="2D

\newtheorem{thm}{Theorem}[section]
\newtheorem{prop}[thm]{Proposition}
\newtheorem{lem}[thm]{Lemma}
\newtheorem{cor}[thm]{Corollary}

\begin{document}

\title{A new near octagon and the Suzuki tower}
\author{Anurag Bishnoi and Bart De Bruyn}
\maketitle

\begin{abstract}
We construct and study a new near octagon of order $(2,10)$ which has its full automorphism group isomorphic to the group $\mathrm{G}_2(4){:}2$ and which contains $416$ copies of the Hall-Janko near octagon as full subgeometries. Using this near octagon and its substructures we give geometric constructions of the $\mathrm{G}_2(4)$-graph and the Suzuki graph, both of which are strongly regular graphs contained in the Suzuki tower.  As a subgeometry of this octagon we have discovered another new near octagon, whose order is $(2,4)$.

\end{abstract}

\section{Introduction and overview} \label{sec1}

A \textit{near $2d$-gon} with $d \in \N$ is a partial linear space $\mS$ that satisfies the following properties: 
\begin{enumerate} [(NP1)]
\item The collinearity graph of $\mathcal{S}$ is connected and has diameter $d$. 
\item For every point $x$ and every line $L$ there exists a unique point $\pi_L(x)$ incident with $L$ that is nearest to $x$. 
\end{enumerate} 
A {\em near polygon} is a near $2d$-gon for some $d \in \N$. Near polygons were introduced by Shult and Yanushka in \cite{Sh-Ya} as geometries related to certain line systems in Euclidean spaces. The class of near $4$-gons coincides with the class of possibly degenerate generalized quadrangles. Generalized quadrangles belong to the family of generalized polygons, an important class of point-line geometries introduced by Jacques Tits in \cite{Ti}. 
It is well known that every generalized $2d$-gon is a near $2d$-gon. 

A near polygon is said to have {\em order} $(s,t)$ if it has  $s+1$ points on each line and $t+1$ lines through each point. 
An important class of near polygons with an order is that of the regular near polygons, which are related to the distance regular graphs (see Section 6.4 of \cite{BCN}). 
A famous example of a regular near octagon ($d = 4$) is the near octagon $\HJ$ \cite{Co} associated with the sporadic simple group $\mathrm{J}_2$ of Hall and Janko. 
It can be constructed by taking the $315$ central involutions of the group $\mathrm{J}_2$ as points  and  the three element subsets $\{x,y,xy\}$ of points where the involutions $x$ and $y$ commute as lines. Therefore, the Hall-Janko near octagon is a \textit{commuting involution graph} of the sporadic simple Hall-Janko group $\mathrm{J}_2$.

Michio Suzuki (cf.~\cite{Su}) constructed a sequence of five finite simple groups $G_0, \dots, G_4$ and five vertex transitive graphs $\G_0, \dots, \G_4$ such that for $i \in \{1, \dots, 4\}$ the graph $\G_{i-1}$ is the local graph of $\G_i$ and the group $G_i$ is an automorphism group of the graph $\G_i$, with vertex stabilizer isomorphic to $G_{i-1}$. 
The largest of these groups was a new sporadic simple group, which is now called $\mathrm{Suz}$.  
The five simple groups are $\mathrm{L}_3(2) < \mathrm{U}_3(3) < \mathrm{J}_2 < \mathrm{G}_2(4) < \mathrm{Suz}$ and the full automorphism group of the graph $\G_i$ is the split extension of $G_i$ by the cyclic group of order $2$, which is denoted as $G_i{:}2$ in the ATLAS notation \cite{Atlas}. 
The graph $\G_0$ is the incidence graph of the unique $2 \mh (7,4,2)$ design (the complementary design of the Fano plane), that is, the quartic vertex transitive co-Heawood graph, and the rest of them are strongly regular graphs having the following parameters $(v,k,\lambda, \mu)$:
\begin{itemize}
\item $\G_1$: $(36, 14, 4, 6)$, the $\mathrm{U}_3(3)$-graph; 

\item $\G_2$: $(100, 36, 14, 12)$, the Hall-Janko graph;

\item $\G_3$: $(416, 100, 36, 20)$, the $\mathrm{G}_2(4)$-graph;

\item $\G_4$: $(1782, 416, 100, 96)$, the Suzuki graph.
\end{itemize}
These sequences of groups and graphs form what Jacques Tits called the {\em Suzuki tower}. 

In this paper, we associate near polygons $\mS_0, \dots, \mS_3$ to the Suzuki tower, where for every $i \in \{1, 2, 3\}$ the near polygon $\mS_{i-1}$ is a full subgeometry of $\mS_{i}$, and for every $i \in \{0, 1, 2, 3\}$ $\mS_i$ has the group $G_i{:}2$ as its full automorphism group. 
Moreover, these near polygons can be used to construct all the graphs in the Suzuki tower.
The near polygons $\mS_0,$ $\mS_1$ and $\mS_2$ are already known and they are respectively isomorphic to the dual of the double of the Fano plane, the dual split Cayley hexagon $\mathrm{H}(2)^D$ and the Hall-Janko near octagon $\HJ$. We construct a new near octagon $\mathcal{G} = \mS_3$ of order $(2,10)$ as a particular kind of \textit{commuting involution graph} of the group $\mathrm{G}_2(4){:}2$.

\begin{thm}\label{main:1}
The point-line geometry $\mathcal{G} = (\mP, \mL)$ formed by taking $\mP$ as the conjugacy class of $4095$ central involutions of the group $G = \mathrm{G}_2(4){:}2$ and lines as the three element subsets $\{x, y, xy\}$ of $\mP$ formed by taking those commuting involutions $x$, $y$ in $\mP$ that satisfy $[G:N_G(\langle x,y \rangle)]\in \{1365, 13650\}$, is a near octagon of order $(2,10)$ with $\mathrm{G}_2(4){:}2$ as its full automorphism group.
The set $S^*$ of lines $\{x, y, xy\}$ with $[G:N_G( \langle x,y \rangle)] = 1365$ is a line spread of this geometry. 
Moreover, this near octagon contains precisely $416$ suboctagons isomorphic to the Hall-Janko octagon as full subgeometries. 
\end{thm}

\bigskip \noindent We note that some other geometries (usually of rank higher than $2$) related to the Suzuki tower have been studied in the literature as well \cite{Ne82, S92, Leemans05}, but to our knowledge none of them include our near octagon $\mG$. 
On the elements of the line spread $S^*$ of $\mG$, a generalized hexagon can be defined as part (1) of the following result shows.

\begin{thm}\label{main{:}2}
\begin{enumerate}
\item[$(1)$] Let $\mathcal{Q}^\ast$ denote the set of all subgeometries of $\mathcal{G}$ isomorphic to the generalized quadrangle $W(2)$ (the so-called quads of $\mathcal{G}$). Then the point-line geometry with point set $S^\ast$, line set $\mathcal{Q}^\ast$ and natural incidence relation (set containment) is isomorphic to the dual split Cayley hexagon $\mathrm{H}(4)^D$.

\item[$(2)$] Let $S \subseteq S^\ast$ and $\mathcal{Q} \subseteq \mathcal{Q}^\ast$ such that $(S,\mathcal{Q})$ is a full subgeometry of $(S^\ast,\mathcal{Q}^\ast) \cong \mathrm{H}(4)^D$ that is a generalized hexagon of order $(4,1)$. 
Let $\mP'$ denote the set of points incident with a line of $S$ and $\mL'$ the set of lines incident with a quad of $\mathcal{Q}$. Then $\mG' = (\mP',\mL')$ is a full subgeometry of $\mathcal{G} = (\mP,\mL)$ that is a near octagon of order $(2,4)$.
\end{enumerate}
\end{thm}

\bigskip \noindent The near octagon $\mG'$ of order $(2,4)$ defined in Theorem \ref{main{:}2}(2) is in fact also a new near octagon. 
From its construction it follows that $\mG'$ has $315$ points and $525$ lines. 
In \cite{adw-hvm}, De Wispelaere and Van Maldeghem showed that the point-line dual of $\HJ$ has a full embedding in the split Cayley hexagon $\mathrm{H}(4)$. By making use of the fact that $(S^\ast,\mathcal{Q}^\ast) \cong \mathrm{H}(4)^D$, we will give another proof for this fact (see Lemma \ref{lem4.16}). 

Our remaining results are regarding constructions of the two largest graphs contained in the Suzuki tower using the near polygon $\mG$ and its subgeometries.

\begin{thm}\label{main:3}
The graph defined by taking those suboctagons of the $\mathrm{G}_2(4)$-near octagon $\mathcal{G}$ that are isomorphic to the Hall-Janko near octagon as vertices and adjacency as intersection in a subhexagon isomorphic to the dual split Cayley hexagon of order $2$, is isomorphic to the $\mathrm{G}_2(4)$-graph. 
\end{thm}

\begin{thm}\label{main:4}
Define a graph as follows:
\begin{itemize}
\item take the elements of $\{\infty\}$, $A$ and $B$ as vertices, where $\infty$ is just a symbol, $A$ is the set of all Hall-Janko suboctagons of the near octagon $\mG$ and $B$ is the line spread $S^*$ of $\mG$;

\item join $\infty$ to all vertices in $A$, join two distinct vertices of $A$ if the corresponding suboctagons intersect in a subhexagon isomorphic to $\mathrm{H}(2)^D$, join a vertex of $A$ to all the vertices in $B$ that correspond to a line intersecting the suboctagon, join two vertices in $B$ if the corresponding lines are at distance $2$ from each other in the near octagon.
\end{itemize}
Then this graph is isomorphic to the Suzuki graph.
\end{thm}

\bigskip \noindent \textbf{Remarks.}
\begin{enumerate}[(1)]
\item The main result of the present paper is the construction of two point-transitive near octagons, one of order $(2,10)$ and another one of order $(2,4)$ which occurs as full subgeometry of the former. In general, it is hard to construct new ``nice near polygons'', for example point-transitive near polygons which are not bipartite graphs. Besides some infinite families (the most recent one being discovered around 15 years ago \cite{Bart03}), there are also some examples related to sporadic simple groups which were discovered around 1980 \cite{As,Br-Co-Ha-Wi,Co,Sh-Ya}.

\item We do not know whether there is a near polygon corresponding to $\mathrm{Suz}$, but certainly the involution geometry of $\mathrm{Suz}$ studied in \cite{Yoshiara88, Bardoe96} is not a near polygon. We can directly see from the suborbit diagram \cite[Fig. 1]{Bardoe96} that there are point-line pairs $(p, L)$ where each point of $L$ is at the same distance from $p$. However, this involution geometry is a near $9$-gon in the sense of \cite[Sec. 6.4]{BCN}. Similarly the involution geometry of the Conway group $\mathrm{Co}_1$ \cite[Fig. 1]{Ba}, which contains $\mathrm{Suz}$, is a near $11$-gon.

\item It was pointed out to us by one of the referees of the follow up paper \cite{ab-bdb:3} that the new near octagon $\mathcal{G}$ can be found in the second layer of the three fold cover of the Suzuki graph discovered by Soicher \cite{S93} (see Remark 1.4 of \cite{ab-bdb:3}). In \cite{ab-bdb:4} we will show that the near octagon $\mathcal{G}'$ can also be found in the second layer of another graph discovered by Soicher in \cite{S93}, namely a three fold cover of the unique strongly regular graph with parameters $(v,k,\lambda,\mu)=(162,56,10,24)$. 

\item The structure of both $\mathcal{G}$ and $\mathcal{G}'$ around a fixed point can be described by a diagram (like in Figure \ref{fig:sub1}, with possible different numbers around the big nodes). In \cite{ab-bdb:4}, we also give a common treatment for those near octagons whose local structure is described by such a diagram.
\end{enumerate}

\bigskip Our paper is organized as follows. In Section \ref{sec3}, we prove that $\mathcal{G}$ is a near octagon and discuss the local structure of $\mathcal{G}$ with respect to a fixed point. This structure will be described by means of a so-called suborbit diagram. In Section \ref{sec4} we prove several geometrical properties of $\mathcal{G}$. The properties related to the line spread $S^*$ and the quads of $\mathcal{G}$ are discussed in Subsection \ref{sec4.1} where we also prove Theorem \ref{main{:}2}. In Subsection \ref{sec4.2} we classify all suboctagons of $\mathcal{G}$ that are isomorphic to the Hall-Janko near octagon $\HJ$, and in Subsection \ref{sec4.3} we determine the full automorphism group of $\mathcal{G}$, hereby completing the proof of Theorem \ref{main:1}. With the help of the derived geometrical properties we then prove Theorems \ref{main:3} and \ref{main:4} in Section \ref{sec5}.

Our initial explorations of the properties of the new near octagon happened in a (computer) model for $\mathcal{G}$ which was quite different from the one given in Theorem \ref{main:1}. We first constructed the new near octagon as a subgeometry of the so-called \textit{valuation geometry} of $\HJ$. Examination of the properties\footnote{Among other things, we computed the full automorphism group with the aid of SAGE \cite{Sage}, and checked that its derived subgroup is a simple group of order $|\mathrm{G}_2(4)|$ that has index $2$ in the full automorphism group, see \cite{ab-bdb:2}.} of this new near octagon showed that the full automorphism group was most likely isomorphic to $\mathrm{G}_2(4){:}2$. Looking for a model of $\mathcal{G}$ in the same spirit as the model for $\HJ$ (using involutions) was successful and ultimately resulted in the (more symmetric) description given in Theorem \ref{main:1}. In appendix \ref{appA}, we discuss the first model we had for this new near octagon, and show that it is indeed isomorphic to the model presented in Theorem \ref{main:1}. 

\section{Preliminaries} \label{sec2}

A {\em point-line geometry} $\mathcal{S}$ is a triple $(\mP,\mL,\mI)$, where $\mP$ is the non-empty point set, $\mL$ is the line set and $\mI \subseteq \mP \times \mL$ is the incidence relation. $\mathcal{S}$ is called a {\em partial linear space} if any pair of distinct points is incident with at most one line. In this case we may identify each line with the subset of points it is incident with and replace $\mI$ with set inclusion. By abuse of notation, we then denote $\mathcal{S}$ by $(\mP, \mL)$. The distance between two points $x_1$ and $x_2$ of a point-line geometry $\mathcal{S}$ will always be measured in its collinearity graph. This distance will be denoted by $\dist_{\mS}(x_1,x_1)$, or shortly by $\dist(x_1,x_2)$ if no confusion could arise. 
If $\Gamma$ is the collinearity graph of a geometry $\mS$, $x$ a point of $\mS$ and $i \in \N$, then $\Gamma_i(x)$ denotes the set of points of $\mathcal{S}$ at distance $i$ from $x$. If $X_1$ and $X_2$ are two nonempty sets of points, then $\dist(X_1,X_2)$ denotes the minimal distance between a point of $X_1$ and a point of $X_2$. If $X$ is a nonempty set of points and $i \in \N$, then $\Gamma_i(X)$ denotes the set of points at distance $i$ from $X$. 
Two points (resp. lines) of a point-line geometry $\mS$ are called \textit{opposite} if they have the maximum possible distance between them in the collinearity graph of $\mS$. 

Let $X$ be a nonempty set of points of a partial linear space $\mS$. It is called a \textit{subspace} if every line meeting $X$ in at least two points is completely contained in $X$. It is called \textit{geodetically closed} or \textit{convex} if every point on every shortest path between two points of $X$ is contained in $X$.  A \textit{line spread} of $\mS$ is a set of pairwise disjoint lines such that every point of $\mS$ lies on at least one of these lines. 
An \textit{ovoid} of $\mS$ is a set of pairwise noncollinear points of $\mS$ such that every line of $\mS$ contains at least one of these points.  

A point-line geometry $\mS = (\mP, \mL, \mI)$ is a \textit{subgeometry} of another point-line geometry $\mS' = (\mP ', \mL ', \mI')$ if $\mP \subseteq \mP'$, $\mL \subseteq \mL '$ and $\mI = \mI ' \cap (\mP  \times \mL)$. A subgeometry is called \textit{full} if for every line $L$ in $\mL$ the set $\{x \in \mP : x ~\mI~ L\}$ is equal to $\{x \in \mP' : x ~\mI'~ L\}$. If $d_{\mS}(x,y) =  d_{\mS'}(x,y)$ for every two points $x,y$ in $\mP$, then we will say that $\mS$ is \textit{isometrically embedded} into $\mS'$. 

A {\em quad} of a near polygon is a convex subspace $Q$ of diameter 2 such that the full subgeometry determined by those points and lines that are contained in $Q$ is a nondegenerate generalized quadrangle. 
In \cite[Prop.~2.5]{Sh-Ya}, it was shown that if $a$ and $b$ are two points of a near polygon at distance $2$ from each other, and if $c$ and $d$ are two common neighbours of $a$ and $b$ such that at least one of the lines $ac$, $ad$, $bc$, $bd$ contains at least three points, then $a$ and $b$ are contained in a unique quad. This quad coincides with the smallest convex subspace containing $a$, $b$ and consists of all points of the near polygon which have distance at most $2$ from $a$, $b$, $c$ and $d$. A point-quad pair $(x,Q)$ in a near polygon is called \textit{classical} if there exists a unique point $x'$ in $Q$ such that $\dist(x,y) = \dist(x,x') + \dist(x', y)$ for all $y$ in $Q$, and \textit{ovoidal} if the points of $Q$ that are nearest to $x$ form an ovoid of $Q$. A quad $Q$ is called {\em classical} if the pair $(x,Q)$ is classical for every point $x$. In \cite[Proposition 2.6]{Sh-Ya}, it was shown that in a near polygon in which each line is incident with at least three points, every point-quad pair is either classical or ovoidal.
We will be using these basic results on near polygons and their quads in our proofs without making an explicit reference. 

A near $2d$-gon $\mathcal{S}$ with $d \geq 2$ is called {\em regular} if it has order $(s,t)$ and there exist constants $t_i$, $i \in \{0, \ldots, d\}$, such that for every pair of points $x$ and $y$ at distance $i$, there are precisely $t_i + 1$ lines through $y$ containing a point at distance $i-1$ from $x$. Clearly, $t_0 = -1$, $t_1 = 0$ and $t_d = t$. We will say that $\mathcal{S}$ is regular with parameters $(s, t;t_2, \ldots, t_{d-1})$. A generalized $2d$-gon is a regular near $2d$-gon with parameters $(s,t;0, \ldots, 0)$. The Hall-Janko near octagon is a regular near octagon with parameters $(2,4;0,3)$. In fact it is the unique regular near octagon with those parameters, as proved by Cohen and Tits in \cite{Co-Ti}. 

For all the group theoretical notations we refer to the ATLAS \cite{Atlas}. An involution $a$ of a group $G$ is called \textit{central} if there exists a Sylow $2$-subgroup $H$ of $G$ such that $a \in C_G(H)$, or equivalently if the centralizer of $a$ contains a Sylow $2$-subgroup. It is well known that the group $\mathrm{G}_2(4){:}2$ has $\mathrm{J}_2{:}2$ as a maximal subgroup of index $416$ and the group $\mathrm{J}_2{:}2$ has $G_2(2)$ as a maximal subgroup of index $100$. The conjugacy class of central involutions of the groups $\mathrm{G}_2(4){:}2$, $\mathrm{J}_2{:}2$, and $G_2(2)$, are all denoted by the symbol $2A$ in the ATLAS. If $H \cong G_2(2)$, $K \cong \mathrm{J}_2{:}2$ and $G \cong \mathrm{G}_2(4){:}2$ are such that $H < K < G$ and if $\Sigma_H$, $\Sigma_K$, $\Sigma_G$ denote the corresponding conjugacy classes of central involutions, then $\Sigma_H = H \cap \Sigma_K = H \cap \Sigma_G$ and $\Sigma_K = K \cap \Sigma_G$.
 Moreover, $\Sigma_H \subseteq H' \cong \mathrm{U}_3(3)$ (the derived subgroup), $\Sigma_K \subseteq K' \cong \mathrm{J}_2$ and $\Sigma_G \subseteq G' \cong \mathrm{G}_2(4)$. Since $\Sigma_G$ generates a normal subgroup of the simple group $G'$, we necessarily have $\langle \Sigma_G \rangle = G'$.
  Similarly, $\langle \Sigma_H \rangle = H'$ and $\langle \Sigma_K \rangle = K'$.

There exists a natural bijective correspondence between the subgroups of $G = \mathrm{G}_2(4){:}2$ isomorphic to $\mathrm{J}_2{:}2$ and the subgroups isomorphic to $\mathrm{J}_2$. Every subgroup isomorphic to $\mathrm{J}_2{:}2$ contains a unique $\mathrm{J}_2$-subgroup, namely its derived subgroup. Conversely, from ATLAS information we see that every $\mathrm{J}_2$-subgroup $K$ of $G$ must be contained in a (maximal) subgroup isomorphic to $\mathrm{J}_2{:}2$, and such a maximal is uniquely determined by $K$ as it necessarily coincides with the normalizer of $K$ inside $G$. 

\section{Construction of the new near octagon} 
\label{sec3}

The group $G = \mathrm{G}_2(4){:}2$ has precisely three conjugacy classes of involutions. The class $2A$ consists of $4095$ involutions all of which are central and contained in the derived subgroup $G' \cong \mathrm{G}_2(4)$. A computer model of the group $G$ can be easily constructed using the computer programming language GAP \cite{Gap}. All claims of the present section have been verified using such a computer model, see \cite{ab-bdb:2}.  

Let $\mP$ denote the elements of the class $2A$ and $\omega$ a fixed element in $\mP$. The group $G$ acts transitively on $\mP$ by conjugation. Let $G_\omega$ be the stabilizer of $\omega$ under this action. Then the action of $G_\omega$ on $\mP$ partitions $\mP$ into eight orbits, which we refer to as \textit{suborbits} of the action of $G$ on $\mP$. The suborbits are of sizes $1$, $2$, $20$, $40$, $320$, $640$, $1024$, $2048$, and we label them as $\mO_0$, $\mO_{1a}$, $\mO_{1b}$, $\mO_{2a}$, $\mO_{2b}$, $\mO_{3a}$, $\mO_{3b}$, $\mO_{4}$, respectively.
 We put $\mathcal{O}_i := \mathcal{O}_{ia} \cup \mathcal{O}_{ib}$ for $i \in \{1,2,3\}$. There are precisely $62$ elements of $\mP \setminus \{ \omega \}$ that commute with $\omega$ and they lie in the suborbits $\mO_{1a}$, $\mO_{1b}$ and $\mO_{2a}$. If $\mO$ is one of these three suborbits and $x$ belongs to $\mO$ then the product $x \omega$ also belongs to $\mO$. 
These three suborbits can also be characterized by the index of the normalizer of the subgroup generated by $\omega$ and an element $x$ of these suborbits. 
The indices for $\mO_{1a}$, $\mO_{1b}$ and $\mO_{2a}$ are $1365$, $13650$ and $27300$, respectively. Therefore, if we take $\mP$ and $\mL$ to be the sets of points and lines as defined in Theorem \ref{main:1} then the points collinear with $\omega$ are in $\mO_{1a}$ and $\mO_{1b}$. In fact we get the \textit{suborbit diagram} drawn in Figure~\ref{fig:sub1} which explains the choice of notation for these suborbits. Suborbits containing central involutions at distance $i$ from $\omega$ in the geometry have been labelled $\mO_{i*}$ where $*$ is $a$, $b$ or void.

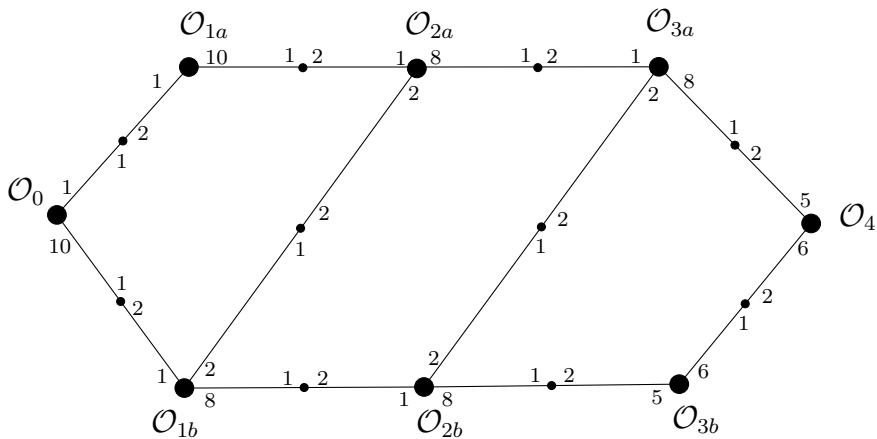
\begin{figure}
\begin{center}
\begin{tikzpicture}[line cap=round,line join=round,>=triangle 45,x=1.0cm,y=1.0cm]
\draw (0.3393083836366923,3.0573010907233384)-- (-1.39563107885606,1.0960651766010965);
\draw (-1.39563107885606,1.0960651766010965)-- (0.28,-1.2);
\draw (0.28,-1.2)-- (3.356594405363218,3.0573010907233384);
\draw (0.3393083836366923,3.0573010907233384)-- (3.356594405363218,3.0573010907233384);
\draw (3.356594405363218,3.0573010907233384)-- (6.52,3.06);
\draw (6.52,3.06)-- (8.523696717569893,0.9829169507863518);
\draw (0.28,-1.2)-- (3.432026555906381,-1.1857573773295889);
\draw (3.432026555906381,-1.1857573773295889)-- (6.52,3.06);
\draw (3.432026555906381,-1.1857573773295889)-- (6.788757255077141,-1.1480413020580074);
\draw (6.788757255077141,-1.1480413020580074)-- (8.523696717569893,0.9829169507863518);
\draw (3.03295311787782,3.933501017569001) node[anchor=north west] {$\mathcal{O}_{2a}$};
\draw (6.20,3.9669953715641197) node[anchor=north west] {$\mathcal{O}_{3a}$};
\draw (8.743740474045596,1.3879301139399627) node[anchor=north west] {$\mathcal{O}_4$};
\draw (6.55,-1.30) node[anchor=north west] {$\mathcal{O}_{3b}$};
\draw (3.183677710855855,-1.35) node[anchor=north west] {$\mathcal{O}_{2b}$};
\draw (-0.30,-1.35) node[anchor=north west] {$\mathcal{O}_{1b}$};
\draw (0.06870278930979522,3.933501017569001) node[anchor=north west] {$\mathcal{O}_{1a}$};
\draw (-2.2,1.722873653891152) node[anchor=north west] {$\mathcal{O}_0$};
\draw (1.45,3.45) node[anchor=north west] {{\scriptsize 1}};
\draw (4.55,3.45) node[anchor=north west] {{\scriptsize 1}};
\draw (7.29,2.50) node[anchor=north west] {{\scriptsize 1}};
\draw (7.42,-0.11) node[anchor=north west] {{\scriptsize 1}};
\draw (4.75,0.92) node[anchor=north west] {{\scriptsize 1}};
\draw (1.59,0.88) node[anchor=north west] {{\scriptsize 1}};
\draw (-0.76,2.05) node[anchor=north west] {{\scriptsize 1}};
\draw (-0.76,0.42) node[anchor=north west] {{\scriptsize 1}};
\draw (1.42,-0.83) node[anchor=north west] {{\scriptsize 1}};
\draw (4.68,-0.80) node[anchor=north west] {{\scriptsize 1}};
\draw (-0.48,2.42) node[anchor=north west] {{\scriptsize 2}};
\draw (1.81,3.44) node[anchor=north west] {{\scriptsize 2}};
\draw (4.90,3.43) node[anchor=north west] {{\scriptsize 2}};
\draw (7.58,2.15) node[anchor=north west] {{\scriptsize 2}};
\draw (-0.55,0.10) node[anchor=north west] {{\scriptsize 2}};
\draw (1.90,1.35) node[anchor=north west] {{\scriptsize 2}};
\draw (5.04,1.28) node[anchor=north west] {{\scriptsize 2}};
\draw (1.88,-0.83) node[anchor=north west] {{\scriptsize 2}};
\draw (5.12,-0.80) node[anchor=north west] {{\scriptsize 2}};
\draw (7.73,0.25) node[anchor=north west] {{\scriptsize 2}};
\draw (-1.65,0.92) node[anchor=north west] {{\scriptsize 10}};
\draw (-1.48,1.70) node[anchor=north west] {{\scriptsize 1}};
\draw (0.14,3.1) node[anchor=north east] {{\scriptsize 1}};
\draw (0.41,2.95) node[anchor=south west] {{\scriptsize 10}};
\draw (2.91,3.41) node[anchor=north west] {{\scriptsize 1}};
\draw (3.37,3.42) node[anchor=north west] {{\scriptsize 8}};
\draw (3.08,2.95) node[anchor=north west] {{\scriptsize 2}};
\draw (6.23,2.93) node[anchor=north west] {{\scriptsize 2}};
\draw (8.20,0.89) node[anchor=north west] {{\scriptsize 6}};
\draw (8.23,1.53) node[anchor=north west] {{\scriptsize 5}};
\draw (6.70,3.11) node[anchor=north west] {{\scriptsize 8}};
\draw (6.00,3.43) node[anchor=north west] {{\scriptsize 1}};
\draw (6.89,-0.74) node[anchor=north west] {{\scriptsize 6}};
\draw (6.28,-1.09) node[anchor=north west] {{\scriptsize 5}};
\draw (3.52,-1.11) node[anchor=north west] {{\scriptsize 8}};
\draw (3.34,-0.57) node[anchor=north west] {{\scriptsize 2}};
\draw (2.95,-1.11) node[anchor=north west] {{\scriptsize 1}};
\draw (-0.23,-0.81) node[anchor=north west] {{\scriptsize 1}};
\draw (0.40,-0.71) node[anchor=north west] {{\scriptsize 2}};
\draw (0.40,-1.12) node[anchor=north west] {{\scriptsize 8}};
\begin{scriptsize}
\draw [fill=black] (-1.39563107885606,1.0960651766010965) circle (3.5pt);
\draw [fill=black] (0.3393083836366923,3.0573010907233384) circle (3.5pt);
\draw [fill=black] (0.28,-1.2) circle (3.5pt);
\draw [fill=black] (3.356594405363218,3.0573010907233384) circle (1.5pt);
\draw [fill=black] (6.52,3.06) circle (3.5pt);
\draw [fill=black] (8.523696717569893,0.9829169507863518) circle (3.5pt);
\draw [fill=black] (3.432026555906381,-1.1857573773295889) circle (3.5pt);
\draw [fill=black] (6.788757255077141,-1.1480413020580074) circle (3.5pt);
\draw [fill=black] (3.337736367727427,3.0384430530875477) circle (3.5pt);
\draw [fill=black] (1.8385223756820597,3.0478720719054433) circle (1.5pt);
\draw [fill=black] (4.928868183863713,3.049221526543774) circle (1.5pt);
\draw [fill=black] (7.521848358784946,2.021458475393176) circle (1.5pt);
\draw [fill=black] (7.656226986323517,-0.0825621756358278) circle (1.5pt);
\draw [fill=black] (5.1103919054917615,-1.166899339693798) circle (1.5pt);
\draw [fill=black] (4.97601327795319,0.9371213113352056) circle (1.5pt);
\draw [fill=black] (1.8560132779531906,-1.1928786886647944) circle (1.5pt);
\draw [fill=black] (1.8088681838637135,0.9192215265437739) circle (1.5pt);
\draw [fill=black] (-0.55781553942803,-0.05196741169945174) circle (1.5pt);
\draw [fill=black] (-0.5281613476096838,2.0766831336622174) circle (1.5pt);
\end{scriptsize}
\end{tikzpicture}
\end{center}
\caption{The suborbit diagram} 
\label{fig:sub1}
\end{figure}

In the literature, suborbit diagrams for finite simple groups where adjacency (in the collinearity graph of the involution geometry) is defined by commutativity have been studied \cite{Bardoe96, Ba}. For drawing the suborbit diagram we have used similar conventions as in \cite{Bardoe96}. In our case adjacency involves both commutativity and a condition on the index of certain normalizers. 
Each of the eight big nodes of the diagram denotes a suborbit and an edge between two such nodes denotes that there is a line that intersects both suborbits.  A smaller node on each edge denotes a line and the two accompanying numbers denote the number of points of the line that lie in the suborbits it intersects. Each number on a big node denotes the number of lines through a given point in that suborbit going to another suborbit. 

We would be using the suborbit diagram in most of our arguments where by ``suborbit diagram with respect to $x$'' we would mean that $\omega = x$ and all the suborbits are defined by the centralizer of the involution $x$ of $\mathrm{G}_2(4){:}2$. If we wish to explicitly indicate the involution $x$ with respect to which the suborbits are considered, we will use the notations $\mathcal{O}_0(x), \mathcal{O}_{1a}(x),\ldots,\mathcal{O}_4(x)$.

\medskip
\begin{thm}
The point-line geometry $\mathcal{G} = (\mP, \mL)$ formed by taking $\mP$ as the conjugacy class of $4095$ central involutions of the group $G = \mathrm{G}_2(4){:}2$ and lines as the three element subsets $\{x, y, xy\}$ of those commuting involutions $x$, $y$ in $\mP$ that satisfy the condition $[G : N_G(\langle x,y \rangle)] \in \{1365, 13650\}$ is a near octagon of order $(2,10)$.
\end{thm}
\begin{proof}
Let $x$ be a fixed central involution of $G$, i.e., a point of $\mP$. It is clear from the suborbit diagram that every other involution is at distance at most $4$ from $x$. 
Therefore the point-line geometry is connected and has diameter $4$. Now let $L$ be any line, then from the suborbit diagram there exists an $i \in \{0,1,2,3\}$ such that $L$ intersects $\mO_{i}$ in one point and $\mO_{i+1}$ in two points. Therefore there exists a unique point on $L$ nearest to $x$. Since there are exactly $22$ neighbours of $x$, and since the automorphism group acts transitively on points, we get that the point-line geometry is a near octagon of order $(2,10)$. 
\end{proof}

\bigskip \noindent \textbf{Remarks.} 
\begin{enumerate}[(1)]
\item Among the conjugacy classes of involutions, the one consisting of the central involutions has the smallest size, namely $4095 = 3 \cdot 1365$. So, among the conjugacy classes of subgroups of Type $C_2 \times C_2$, there exists only one of size 1365, namely the one consisting of all subgroups of the form $\langle x,y \rangle$, where $x$ and $y$ are two distinct central involutions satisfying $[G : N_G(\langle x,y \rangle)] =1365$. This class consists of the 1365 long root subgroups of $\mathrm{G}_2(4)$.

\item Unlike the group $\mathrm{G}_2(4){:}2$, the group $\mathrm{J}_2{:}2$ has only one orbit on the pairs $(x,y)$ of distinct commuting central involutions. If $H$ is a subgroup of $\mathrm{G}_2(4){:}2$ isomorphic to $\mathrm{J}_2{:}2$ and $x,y$ are two distinct commuting central involutions of $H$, then $[G : N_G(\langle x,y \rangle)] = 13650$.
\end{enumerate}
\section{Properties of the new near octagon} \label{sec4}

In this section, we derive several properties of the near octagon $\mathcal{G}$. All of these will be derived from the information provided in Section \ref{sec3}.

\subsection{A line spread and the quads of the near octagon} \label{sec4.1}

\begin{lem} \label{lem4.1}
The set $S^\ast$ which consists of all lines of the form $\{x, y, xy\}$ where $x,y$ are two distinct central involutions of $G = \mathrm{G}_2(4){:}2$ satisfying $[G:N_G( \langle x,y \rangle)] = 1365$ is a line spread of $\mathcal{G}$. 
\end{lem}
\begin{proof}
By the suborbit diagram, every central involution is contained in a unique element of $S^\ast$, implying that $S^\ast$ is a line spread of $\mathcal{G}$. 
\end{proof}

\medskip \noindent If $x$ is a point of $\mathcal{G}$, then we denote by $L_x$ the unique line through $x$ belonging to $S^\ast$. Clearly, $L_x = \{ x \} \cup \mathcal{O}_{1a}(x)$.

\medskip \noindent Our next aim is to determine all quads of $\mathcal{G}$. 
We prove that all these quads are isomorphic to the generalized quadrangle $W(2)$, which is the unique generalized quadrangle of order $(2,2)$ (cf. \cite[5.2.3]{Pa-Th}). 

\begin{lem} \label{lem4.2}
Let $x$ and $y$ be two points of $\mathcal{G}$ at distance $2$ from each other. Then $x$ and $y$ are contained in a quad if and only if $y \in \mathcal{O}_{2a}(x)$. If this is the case, then the unique quad through $x$ and $y$ is isomorphic to $W(2)$ and contains the line $L_x$.
\end{lem}
\begin{proof}
The points $x$ and $y$ are contained in a quad if and only if $x$ and $y$ have at least two common neighbors. By the suborbit diagram, we know that this happens precisely when $y \in \mathcal{O}_{2a}(x)$. If $y \in \mathcal{O}_{2a}(x)$, then $x$ and $y$ have precisely three common neighbors, showing that the unique quad through $x$ and $y$ has order $(2,2)$, necessarily isomorphic to $W(2)$. Moreover, one of the three common neighbors of $x$ and $y$ lies in $\mathcal{O}_{1a}(x)$, showing that $L_x$ is contained in the quad.
\end{proof}

\begin{cor} \label{co4.3}
A line of $S^\ast$ and a quad of $\mathcal{G}$ can never meet in a point. So, the lines of $S^\ast$ contained in a given quad determine a spread of this quad.
\end{cor}

\begin{lem} \label{lem4.4}
Let $x$ be a point of $\mathcal{G}$ and $M$ a line through $x$ distinct from $L_x$. Then $M$ is contained in a unique quad. This quad contains the line $L_x$. 
\end{lem}
\begin{proof}
Note first that there is at most one quad through $M$. 
Indeed, by Corollary \ref{co4.3}, every quad through $M$ must also contain the line $L_x$, and we know that there is at most one quad through two distinct intersecting lines in a near polygon. 

In the suborbit diagram (see Figure~\ref{fig:sub1}) let $\mO_0$ be $\{x\}$. Let $x'$ be a point in $M \setminus x$. Then $x'$ lies in the suborbit $\mO_{1b}$, and hence it has two lines through it that meet the suborbit $\mO_{2a}$. Let $y$ be a point in $\mO_{2a}$ that is collinear with $x'$. By Lemma \ref{lem4.2}, there is a unique quad containing $x$, $y$ and the line $L_x$. As $x'$ is a common neighbor of $x$ and $y$, the point $x'$ and the line $M$ are also contained in this quad.
\end{proof}

\begin{cor} \label{co4.5}
Every point $x$ of $\mathcal{G}$ is contained in five quads. Each of these five quads contains the line $L_x$.
\end{cor}
\begin{proof}
As $\mG$ is of order $(2, 10)$ there are $10$ lines through $x$ distinct from $L_x$, and each of these lines is contained in a unique quad (necessarily containing $L_x$). This gives rise to 10 quads through $x$, but each of them is counted twice as each of them contains two lines through $x$ distinct from $L_x$.
\end{proof}

\begin{cor} \label{co4.6}
Two central involutions of $G$ are contained in a quad if and only if they commute.
\end{cor}
\begin{proof}
If $x$ is a central involution, then the central involutions which commute with $x$ are those contained in $\mO_0(x) \cup \mO_{1a}(x) \cup \mO_{1b}(x) \cup \mO_{2a}(x)$. The claim then follows from Lemmas \ref{lem4.2} and \ref{lem4.4}.
\end{proof}

\begin{lem} \label{lem4.7}
If $Q$ is a quad and $L$ a line of the spread $S^\ast$ containing a point at distance $i$ from $Q$, then $i \in \{ 0,1,2 \}$ and $L$ is completely contained in $\Gamma_i(Q)$. 
\end{lem}
\begin{proof}
We may suppose that $i = \dist(L,Q)$. Let $x \in L$ and $y \in Q$ such that $\dist(x,y)=i$. By Corollary \ref{co4.3}, $L_y$ is contained in $Q$. 

If $i=0$, then $L$ is completely contained in $Q$ by Corollary \ref{co4.3}. So, we may suppose that $i \geq 1$.

Suppose $i=1$. Then $L_x = L$ is disjoint from $Q$ by Corollary \ref{co4.3}.  By Lemma \ref{lem4.4}, there exists a unique quad through $xy$. Since this quad contains the lines $L_x$ and $L_y$, every point of $L_x$ is collinear with a unique point of $L_y$ (which belongs to $Q$). 

Suppose $i=2$. Let $z \in \G_1(Q)$ be a common neighbor of $x$ and $y$. The unique quad through $L_z \subseteq \Gamma_1(Q)$ and $zx$ contains the line $L_x$, showing that every point of $L_x$ is collinear with a point of $L_z \subseteq \Gamma_1(Q)$, which implies that every point of $L_x$ has distance at most and hence precisely 2 from $Q$.

Suppose $i \geq 3$. Since the distance from a point to $Q$ is at most 3, we must have that $L \subseteq \Gamma_3(Q)$. Every point of $L$ must then be ovoidal with respect to $Q$. By (NP2), every point of $Q$ lies at distance 3 from a unique point of $L$, showing that the three ovoids of $Q$ determined by the points of $L$ form a partition of $Q$. This is however impossible, as the generalized quadrangle $W(2)$ has no partition in ovoids.
\end{proof}

\begin{lem} \label{lem4.8}
Every point-quad pair in $\mathcal{G}$ is classical. 
\end{lem}
\begin{proof}
Let $Q$ be a quad of $\mathcal{G}$. By Lemma \ref{lem4.7}, every point has distance at most $2$ from $Q$. Every point at distance at most $1$ from a quad of a general near polygon is classical with respect to that quad. In particular, every point of $Q \cup \G_1(Q)$ is classical with respect to $Q$. Every point of $\G_1(Q)$ is collinear with a unique point of $Q$, implying that $|\G_1(Q)| = |Q| \cdot (11-3) \cdot 2 = 240$ and $|\G_2(Q)| = 4095 - |\G_1(Q)| - |Q| = 3840$.

Let $x$ be a point in $\G_1(Q)$ and $x'$ the unique neighbour of $x$ in $Q$. If $L$ is a line through $x$ contained in $\Gamma_1(Q)$ then the points of $Q$ collinear with a point of $L$ form a line $L'$. The lines $L$ and $L'$ are contained in a quad, which necessarily coincides with the unique quad $Q'$ through the line $xx'$. So, the line $L' = Q \cap Q'$ should coincide with $L_{x'}$.

So, through $x$ there is a unique line meeting $Q$, two lines contained in $\Gamma_1(Q)$ and eight lines meeting $\Gamma_2(Q)$ (necessarily in two points). Since $|\G_1(Q)| \cdot 8 \cdot 2 = 3840 = |\G_2(Q)|$, we must have that every point of $\G_2(Q)$ is collinear with a unique point of $\G_1(Q)$ and hence at distance $2$ from a unique point in $Q$.  
All point-quad pairs must therefore be classical.
\end{proof}

\bigskip \noindent The relation defined on pairs $(x,y) \in \mathcal{P} \times \mathcal{P}$ by the condition $y \in \mathcal{O}_{2a}(x)$ is symmetric. Indeed, $y \in \mathcal{O}_{2a}(x)$ if and only if $\dist(x,y)=2$ and $x,y$ commute (regarded as involutions). (The fact that the relation is symmetric also follows from Lemma \ref{lem4.2}.) The following lemma shows that also the relation defined on pairs $(x,y) \in \mathcal{P} \times \mathcal{P}$ by the condition $y \in \mathcal{O}_{3a}(x)$ is symmetric.

\begin{lem} \label{lem4.9}
Let $x$ and $y$ be two points of $\mathcal{G}$ at distance $3$ from each other. Then the following are equivalent:
\begin{enumerate}
\item[$(1)$] $y \in \mathcal{O}_{3a}(x)$;
\item[$(2)$] $x \in \mathcal{O}_{3a}(y)$;
\item[$(3)$] there is a quad through $x$ meeting a line through $y$;
\item[$(4)$] there is a quad through $y$ meeting a line through $x$.
\end{enumerate}
\end{lem}
\begin{proof}
By the suborbit diagram and Lemma \ref{lem4.2}, (1) and (3) are equivalent, as well as (2) and (4). So, by symmetry it suffices to show that (3) implies (4). Suppose $Q$ is a quad through $x$ and $L$ is a line through $y$ meeting $Q$ in a point $z$. The line $L_z$ is contained in $Q$ and contains a point $u$ collinear with $x$. The unique quad through $L=zy$ contains the line $L_z$ and meets the line $xu$ through $x$.  
\end{proof}

\begin{lem} \label{lem4.10}
One of the following cases occurs for a line $L \in S^\ast$:
\begin{enumerate}
\item[$(1)$] $L$ meets the suborbits $\mathcal{O}_0$ and $\mathcal{O}_{1a}$;
\item[$(2)$] $L$ meets the suborbits $\mathcal{O}_{1b}$ and $\mathcal{O}_{2a}$;
\item[$(3)$] $L$ meets the suborbits $\mathcal{O}_{2b}$ and $\mathcal{O}_{3a}$;
\item[$(4)$] $L$ meets the suborbits $\mathcal{O}_{3b}$ and $\mathcal{O}_4$.
\end{enumerate}
\end{lem}
\begin{proof}
If $L$ meets $\mathcal{O}_0 \cup \mathcal{O}_{1a}$, then necessarily $L= L_\omega$ and so $L$ meets the suborbits $\mathcal{O}_0$ and $\mathcal{O}_{1a}$. So, we may suppose that $L$ is disjoint from $\mathcal{O}_0 \cup \mathcal{O}_{1a}$.

Suppose $L$ contains a point $x \in \mathcal{O}_{1b}$. Then the unique quad through the line $\omega x$ contains $L$ and is completely contained in $\mathcal{O}_0 \cup \mathcal{O}_{1a} \cup \mathcal{O}_{1b} \cup \mathcal{O}_{2a}$, showing that $L$ meets $\mathcal{O}_{1b}$ and $\mathcal{O}_{2a}$. On the other hand, suppose that $L$ contains a point $y \in \mathcal{O}_{2a}$. Then $\omega$ and $y$ are contained in a unique quad which contains the line $L$. Since $Q \subseteq \mathcal{O}_0 \cup \mathcal{O}_{1a} \cup \mathcal{O}_{1b} \cup \mathcal{O}_{2a}$, $L_\omega = \mathcal{O}_0 \cup \mathcal{O}_{1a}$ and $L_y \cap L_\omega = \emptyset$, we have that $L$ meets $\mathcal{O}_{1b}$ and $\mathcal{O}_{2a}$. In the sequel, we will therefore assume that $L$ is disjoint from $\mathcal{O}_0 \cup \mathcal{O}_{1a} \cup \mathcal{O}_{1b} \cup \mathcal{O}_{2a}$.

Suppose $L$ contains a point $x$ of $\mathcal{O}_{3a}$. By the suborbit diagram, we know that the point $x$ is contained in a unique line $M$ that contains a point $y \in \mathcal{O}_{2a}$. The line $L_y$ meets $\mathcal{O}_{1b}$ in a point $u$ and is distinct from $M$, showing that $M \not\in S^\ast$. The unique quad $Q$ through $M$ contains the lines $L_y$ and $L$. So, $L$ contains a point at distance 2 from $\omega$, namely the point of $L$ collinear with $u \in L_y$. This shows that $L$ meets $\mathcal{O}_{3a}$ and $\mathcal{O}_{2b}$. As $|L \cap \mathcal{O}_{3a}|=2$, there are $\frac{|\mathcal{O}_{3a}|}{2} = \frac{640}{2} = 320$ lines of $S^\ast$ meeting $\mathcal{O}_{3a}$ and $\mathcal{O}_{2b}$. Since also $|\mathcal{O}_{2b}|=320$, we have that every line of $S^\ast$ that meets $\mathcal{O}_{2b}$ should also meet $\mathcal{O}_{3a}$.

The lines of $S^\ast$ which we have not yet considered should all meet $\mathcal{O}_{3b}$ and $\mathcal{O}_4$.
\end{proof}

\bigskip \noindent As before, we denote by $\mathcal{Q}^\ast$ the set of quads of $\mathcal{G}$.

\begin{lem} \label{lem4.11}
The geometry formed by taking the lines in $S^\ast$ as points and the quads in $\mathcal{Q}^\ast$ as lines is isomorphic to $\mathrm{H}(4)^D$.
\end{lem}
\begin{proof}
It is well known that the collinearity graph of $\mathrm{H}(4)^D$ is isomorphic to the graph whose vertices are the 1365 long root subgroups of $\mathrm{G}_2(4) = G'$, with two distinct long root subgroups being adjacent whenever they commute. The lines of $\mathrm{H}(4)^D$ then correspond to the maximal cliques (of size 5) of this collinearity graph. Since every line of $(S^\ast,\mathcal{Q}^\ast)$ contains five points, it thus suffices to prove that the collinearity graphs of $(S^\ast,\mathcal{Q}^\ast)$ and $\mathrm{H}(4)^D$ are isomorphic. The elements of $S^\ast$ are the sets $\{ x,y,xy \}$, where $x$ and $y$ are two distinct central involutions satisfying $[G : N_G(\langle x,y \rangle)] = 1365$, and can be put in bijective correspondence with the long root subgroups of $\mathrm{G}_2(4)$ via the correspondence $\{ x,y,xy \} \leftrightarrow \{ e,x,y,xy \}$ (with $e$ being the identity element of $\mathrm{G}_2(4)$). Now, two central involutions commute if and only if they are contained in a quad (regarded as points of $\mathcal{G}$). So, in view of Corollary \ref{co4.3}, two long root subgroups $\{ e,x,y,xy \}$ and $\{ e,x',y',x'y' \}$ commute if and only if the corresponding lines $\{ x,y,xy \}$ and $\{ x',y',x'y' \}$ of $S^\ast$ are contained in a quad. This  shows that the collinearity graphs of $(S^\ast,\mathcal{Q}^\ast)$ and $\mathrm{H}(4)^D$ are isomorphic.
\end{proof}

\begin{lem} \label{lem4.12}
Let $K$ and $M$ be two lines of $S^\ast$ and let $\delta$ denote the distance between $K$ and $M$ in the generalized hexagon $(S^\ast,\mathcal{Q}^\ast) \cong \mathrm{H}(4)^D$. Then $\dist(K,M)=\delta$ and every point of $K$ has distance $\delta$ from a unique point of $M$.  
\end{lem}
\begin{proof}
Put $\delta' := \dist(K,M)$ and let $x_0,x_1,\ldots,x_{\delta'}$ be a path of length $\delta'$ connecting a point $x_0$ of $K$ with a point $x_{\delta'}$ of $M$. Put $L_i := L_{x_i}$, $i \in \{ 0,1,\ldots,\delta' \}$. If two consecutive lines $L_i$ and $L_{i+1}$ are distinct, then they are contained in the unique quad through $x_ix_{i+1}$ and so every point of $L_i$ is collinear with a unique point of $L_{i+1}$. But since $\delta'$ is the smallest distance between a point of $K$ and a point of $M$, two consecutive lines are $L_i$ and $L_{i+1}$ must be distinct (otherwise we could construct a shorter path), and every point of $K$ has distance $\delta'$ from a necessarily unique point of $M$. So, in order to prove the lemma, it suffices to show that $\delta = \delta'$. Since there exists a path of length $\delta'$ in $(S^\ast,\mathcal{Q}^\ast)$ connecting $K$ and $M$, we have $\delta \leq \delta'$. So, it suffices to show that $\delta \geq \delta'$.

Suppose $U_0,U_1,\ldots,U_\delta$ is a path of length $\delta$ in $(S^\ast,\mathcal{Q}^\ast)$ connecting the lines $U_0=K$ and $U_\delta=M$. For every two consecutive lines $U_i$ and $U_{i+1}$, we know that every point of $U_i$ is collinear with a unique point of $U_{i+1}$, implying that there exists a point of $M$ at distance at most $\delta$ from a point of $K$. This implies that $\delta \geq \delta'$, as we needed to show. 
\end{proof}

\medskip \noindent The following is an immediate corollary of Lemma \ref{lem4.12}.

\begin{cor} \label{co4.12}
Let $x$ and $y$ be two points of $\mathcal{G}$, let $y'$ be the point of $L_y$ nearest to $x$ and let $\delta$ denote the distance between $L_x$ and $L_y$ in the generalized hexagon $(S^\ast,\mathcal{Q}^\ast) \cong \mathrm{H}(4)^D$. Then $\dist(x,y)$ is equal to $\delta$ if $y' = y$ and equal to $\delta+1$ otherwise. 
\end{cor}

\bigskip \noindent The dual split Cayley hexagon $\mathrm{H}(4)^D$ is known to have copies of the unique generalized hexagon of order $(4,1)$ as subgeometries . Each such subhexagon can be used to construct a full subgeometry of $\mathcal{G}$ that is also a (new) near polygon. 
This near polygon has the same order and the same number of points as $\HJ$. 

\begin{lem} \label{lem4.14}
Let $S \subseteq S^\ast$ and $\mathcal{Q} \subseteq \mathcal{Q}^\ast$ such that $(S,\mathcal{Q})$ is a full subgeometry of $(S^\ast,\mathcal{Q}^\ast) \cong \mathrm{H}(4)^D$ that is a generalized hexagon of order $(4,1)$. Let $\mathcal{P}'$ denote the set of points incident with a line of $S$ and let $\mathcal{L}_1$ denote the set of lines incident with a quad of $\mathcal{Q}$. Then $\mG' = (\mathcal{P}',\mathcal{L}')$ is a full subgeometry of $\mathcal{G} = (\mathcal{P},\mathcal{L})$ that is a near octagon of order $(2,4)$.
\end{lem}
\begin{proof}
Let $L \in \mathcal{L}'$. Then $L \subseteq Q$ for a certain quad $Q \subseteq \mathcal{Q}$. 
The lines of $S^\ast$ contained in $Q$ all belong to $S$, implying that every point of $L$ belongs to $\mathcal{P}'$. So, $\mG' = (\mathcal{P}',\mathcal{L}')$ is a full subgeometry of $\mathcal{G} = (\mathcal{P},\mathcal{L})$.

If $x$ and $y$ are two points of $\mathcal{P}'$ which are collinear in $\mathcal{G}$, then the lines $L_x \in S$ and $L_y \in S$ are equal or collinear in $(S^\ast,\mathcal{Q}^\ast)$ and hence also in $(S,\mathcal{Q})$. There must exist a quad $Q \in \mathcal{Q}$ containing $L_x$ and $L_y$. The lines of $S$ contained in $Q$ cover $Q$, showing that every point of the line $xy$ is contained in $\mathcal{P}'$. 
So, $\mathcal{P}'$ is a subspace of $\mathcal{G}$ and the full subgeometry induced on $\mathcal{P}'$ is precisely $\mG'$.

Note that any subhexagon of order $(4,1)$ of $\mathrm{H}(4)^D$ is isometrically embedded into $\mathrm{H}(4)^D$. 
So, if $K$ and $L$ are two lines of $S$, then the distance between $K$ and $L$ is the same in the geometries $(S,\mathcal{Q})$ and $(S^\ast,\mathcal{Q}^\ast)$. Corollary \ref{co4.12} then implies that the distance between two points $x, y \in \mathcal{P}'$ is the same in the geometries $\mathcal{G}$ and $\mG'$. So, Property (NP2) in the definition of near polygon remains valid for $\mG'$. By taking suitable points on opposite lines belonging to $S$, we see that the diameter of $\mG'$ is also 4. So, $\mG'$ is a near octagon.

Every point $x \in \mathcal{P}'$ is contained in two quads of $S$ which intersect in the line $L_x$. So, there are precisely five lines of $\mathcal{L}'$ through $x$, showing that the near octagon $\mG'$ has order $(2,4)$.
\end{proof}

\subsection{The Hall-Janko suboctagons} \label{sec4.2}

In this subsection, we classify all Hall-Janko suboctagons of $\mathcal{G}$. These are (full) subgeometries of $\mathcal{G}$ that are isomorphic to $\HJ$. We will show that there are 416 such subgeometries and that all of them are isometrically embedded into $\mathcal{G}$. In the following lemma, we already construct all these 416 subgeometries from the 416 (maximal) subgroups of $\mathrm{G}_2(4){:}2$ isomorphic to $\mathrm{J}_2{:}2$.

\begin{lem} \label{lem4.15}
\begin{enumerate}
\item[$(1)$] Let $H$ be a (maximal) subgroup of $G = \mathrm{G}_2(4) : 2$ isomorphic to $\mathrm{J}_2{:}2$. Then the set $\Sigma_H$ of central involutions contained in $H$ is a subspace of $\mathcal{G}$ on which the induced subgeometry, denoted by $\mathcal{S}_H$, is a Hall-Janko suboctagon.
\item[$(2)$] If $H_1$ and $H_2$ are two distinct maximal subgroups of $G$ isomorphic to $\mathrm{J}_2{:}2$, then $\mathcal{S}_{H_1}$ and $\mathcal{S}_{H_2}$ are distinct subgeometries.
\end{enumerate}
\end{lem}
\begin{proof}
(1) On the set $\Sigma_H \subseteq H' \cong \mathrm{J}_2$, a Hall-Janko near octagon $\mathcal{S}_H'$ can be defined by taking as lines all the sets $\{ x,y,xy \}$, where $x$ and $y$ are two distinct commuting elements of $\Sigma_H$. Recall that if the elements $x,y \in \Sigma_H$ commute, then $[ G : N_G(\langle x,y \rangle) ] = 13650$, implying that $\{ x,y,xy \}$ is a line of $\mathcal{G}$. Conversely, if $x,y \in \Sigma_H$ such that $\{ x,y,xy \}$ is a line of $\mathcal{G}$, then $x,y$ commute and hence $\{ x,y,xy \}$ is also a line of $\mathcal{S}_H'$.

(2) We need to show that $H$ is uniquely determined by $\Sigma_H$. The subgroup generated by $\Sigma_H$ is a normal subgroup of $H' \cong \mathrm{J}_2$ and hence coincides with $H'$. Inside $G = \mathrm{G}_2(4) {:}2$, there is a unique subgroup isomorphic to $\mathrm{J}_2{:}2$ that contains $H' \cong \mathrm{J}_2$, namely its normalizer. Hence,  $H = N_G( \langle \Sigma_H \rangle)$.
\end{proof}

\bigskip \noindent Before proceeding to prove that every Hall-Janko suboctagon is as described in Lemma \ref{lem4.15}, we first give an alternative proof of a result of \cite{adw-hvm}, stating that the point-line dual $\HJ^D$ of $\HJ$ has a full embedding into the split Cayley hexagon $\mathrm{H}(4)$.

\begin{lem} \label{lem4.16}
The geometry $\HJ^D$ has a full embedding in $\mathrm{H}(4)$.
\end{lem}
\begin{proof}
Let $H$ be a maximal subgroup of $G$ isomorphic to $\mathrm{J}_2{:}2$. 
Then by Lemma \ref{lem4.15}, $\mathcal{S}_H$ is a full subgeometry isomorphic to $\HJ$. Every line $L$ of $\mathcal{S}_H$ is contained in a unique quad $Q_L$ (as $L \not\in S^\ast$, see the final remarks of Section \ref{sec3}). As any two involutions of $Q_L \cap H$ commute, $Q_L \cap H$ is at most a line of $H$, implying that $L = Q_L \cap H$. So, if $L_1,L_2,\ldots,L_5$ are the five lines of $\mathcal{S}_H$ through a given point $x$, then the quads $Q_{L_1},Q_{L_2},\ldots,Q_{L_5}$ are mutually distinct and hence are all the five quads through $L_x$. This implies that the maps $x \mapsto L_x$, $L \mapsto Q_L$ define a full embedding of the dual of $\mathcal{S}_H$ into the dual of $(S^\ast,\mathcal{Q}^\ast)$, which is isomorphic to $\mathrm{H}(4)$.
\end{proof}

\bigskip \noindent In the sequel, $\mathcal{H}$ will denote an arbitrary Hall-Janko suboctagon. We will derive several properties of $\mathcal{H}$ that will enable us to prove that there are at most (and hence precisely) 416 Hall-Janko suboctagons. 

\begin{lem} \label{lem4.17}
If $x$ and $y$ are two points of $\mathcal{H}$ such that $\dist_\mH(x,y) \leq 2$ then $\dist_\mH(x,y) = \dist_\mG(x,y)$. 
\end{lem}
\begin{proof}
Obviously, this is true if $\dist_\mH(x,y) \leq 1$. So, suppose that $\dist_\mH(x,y) = 2$. Then $\dist_\mG(x,y) \leq 2$. If $\dist_\mG(x,y) = 1$ then we would get a triangle which contradicts (NP2). Therefore, $\dist_\mG(x,y) = 2$.
\end{proof}

\begin{lem} \label{lem4.18}
If $x$ and $y$ are two points of $\mathcal{H}$ such that $\dist_\mH(x,y) = \dist_\mG(x,y) = 3$ then $y \in \mO_{3b}(x)$. 
\end{lem}
\begin{proof}
Since $\mH$ is a regular near octagon with parameters $(2,4;0,3)$ there must be four lines of $\mH$ through $y$ that contain a point at distance $2$ from $x$ in $\mathcal{H}$ (and hence also in $\mG$ by Lemma~\ref{lem4.17}). But, by the suborbit diagram, if $y$ lies in $\mO_{3a}(x)$ then there are are only three lines through $y$ containing a point at distance 2 from $x$ in $\mathcal{G}$.
\end{proof}

\begin{lem} \label{lem4.19}
Let $Q$ be a quad of $\mathcal{G}$ and $x$, $y$ two points of $Q$ such that $\dist_\mG(x,y) = 2$. If $z$ is a point collinear with $y$ and not contained in $Q$ then $z \in\mO_{3a}(x)$.
\end{lem}
\begin{proof}
As $z$ is classical with respect to $Q$, $\dist_\mG(x,z)=3$. By Lemma \ref{lem4.9}, $z \in \mO_{3a}(x)$.
\end{proof}

\begin{lem} \label{lem4.20}
A quad $Q$ of $\mathcal{G}$ cannot contain a pair of intersecting lines of $\mathcal{H}$. 
\end{lem}
\begin{proof}
Suppose $L_1$ and $L_2$ are two intersecting lines of $\mathcal{H}$ contained in $Q$. Let $x_1 \in L_1 \setminus L_2$ and $x_2 \in L_2 \setminus L_1$. As there are five lines through $x_2$ contained in $\mH$ and only three contained in $Q$, there exists a neighbour $x_3$ of $x_2$ in $\mH \setminus Q$. For this point $x_3$, we have $\dist_\mH(x_1,x_3)=\dist_\mG(x_1,x_3)=3$. By Lemma~\ref{lem4.19} $x_3 \in \mO_{3a}(x_1)$. This contradicts Lemma~\ref{lem4.18} which would imply that $x_3 \in \mO_{3b}(x_1)$.
\end{proof}

\begin{lem} \label{lem4.21}
None of the lines of the line spread $S^\ast$ is contained in $\mH$.
\end{lem}
\begin{proof}
Suppose $L$ is a line of $S^\ast$ contained in $\mH$, and let $M$ denote any other line of $\mH$ meeting $L$ in a point. By Lemma \ref{lem4.4}, there is a unique quad $Q$ containing $L$ and $M$. This quad would contradict Lemma \ref{lem4.20}.
\end{proof}

\begin{lem} \label{lem4.22}
Every quad which contains a point $x$ of $\mH$ contains a unique line of $\mH$ through $x$. 
\end{lem}
\begin{proof}
By Lemma \ref{lem4.21}, the line $L_x$ is not contained in $\mH$. There are five lines through $x$ contained in $\mH$. By Lemma \ref{lem4.20}, each of the five quads through $L_x$ contains at most one and hence precisely one of these five lines.
\end{proof}

\begin{lem} \label{lem4.23}
$\mathcal{H}$ is isometrically embedded into $\mathcal{G}$.
\end{lem}
\begin{proof}
Suppose $\dist_\mG(x,y) \not= \dist_\mH(x,y)$ for certain points $x$ and $y$ of $\mathcal{H}$, and suppose $x$ and $y$ have been chosen in such a way that $i := \dist_\mH(x,y)$ is as small as possible. By Lemma \ref{lem4.17}, $i \in \{ 3,4 \}$. Let $y'$ be a point in $\mH$ such that $\dist_\mH(x,y')=i-1$, $\dist_\mH(y',y)=1$ and let $y''$ denote the third point on the line $yy'$. By (NP2) we know that $\dist_\mH(x,y'') = \dist_\mH(x,y) = i$. By the minimality of $i$, $\dist_\mG(x,y') = i - 1$. By (NP2), $\{ \dist_\mG(x,y),\dist_\mG(x,y'') \} = \{ i-1,i-2 \}$. So, still under the assumption that the distance $\dist_\mH(x,y)$ is as small as possible, we could have chosen $y$ in such a way that $\dist_\mG(x,y) = \dist_\mH(x,y)-2$.

Suppose $i=3$. Then we can choose $x,y \in \mH$ such that $\dist_\mH(x,y) = 3$ and $\dist_\mG(x,y) = 1$. Let $x$, $z_1$, $z_2$, $y$ be a shortest path between $x$ and $y$ in $\mH$. By Lemma~\ref{lem4.17}, $\dist_\mG(x,z_2) = 2$. Since $x$ and $z_2$ have at least two common neighbours in $\mG$ (namely $y$ and $z_1$), there exists a quad $Q$ containing $x$, $z_2$ and all their common neighbours. The quad $Q$ would then contain the intersecting lines $xz_1$ and $z_1z_2$, which is in contradiction with Lemma~\ref{lem4.20}. 

Therefore $i=4$. Again, we can choose points $x,y \in \mH$ such that $\dist_\mH(x,y) = 4$ and $\dist_\mG(x,y) = 2$. Let $y'$ be a common neighbour of $x$ and $y$ in $\mG$. By Lemma~\ref{lem4.4} there exists a quad $Q$ through the line $xy'$. By Lemma~\ref{lem4.22} $Q$ must contain a line $M$ of $\mH$ through $x$. Let $x'$ be the unique point on $M$ satisfying $\dist_\mH(y,x') = 3$. Since $i=4$, we also have $\dist_\mG(y,x') = 3$. So, $y \in \mathcal{O}_{3b}(x')$ by Lemma \ref{lem4.18}. Since $\dist_\mG(y,x') = 3$, the quad $Q$ cannot contain the point $y$. Lemma \ref{lem4.19} (with $x$, $y$ and $z$ replaced by $x'$, $y'$ and $y$) would then imply that $y \in \mathcal{O}_{3a}(x')$, which is in contradiction with the earlier claim that $y \in \mathcal{O}_{3b}(x')$.
\end{proof}

\begin{lem} \label{lem4.24}
If $x$ and $y$ are two points of $\mathcal{H}$ such that $\dist_\mH(x,y) = 2$ then $y \in \mO_{2b}(x)$.
\end{lem}
\begin{proof}
We also have $\dist_\mG(x,y)=2$. Let $x' \in \mH$ be a common neighbour of $x$ and $y$. If $y \in \mO_{2a}(x)$, then the unique quad through $x$ and $y$ would contain the intersecting lines $xx'$ and $x'y$, which would be in violation with Lemma \ref{lem4.20}. Therefore, $y \in \mO_{2b}(x)$.
\end{proof}

\begin{lem} \label{lem4.25}
Through every pair of opposite points of $\mG$ there is at most one Hall-Janko suboctagon. 
\end{lem}
\begin{proof}
Let $x$ and $y$ be two opposite points of $\mG$ and suppose the Hall-Janko octagon $\mH$ contains $x$ and $y$. We will show that $\mH$ is uniquely determined by $x$ and $y$. In this proof all suborbits are considered with respect to the point $x$. By Lemma \ref{lem4.23}, the distance between two points of $\mH$ is the same in the geometries $\mathcal{H}$ and $\mathcal{G}$. 

There are five lines through $y$ inside $\mH$ that contain a point at distance $3$ from $x$. By Lemma~\ref{lem4.18} all of these lines must intersect $\mO_{3b}$. By the suborbit diagram and Lemma \ref{lem4.10}, there are exactly six such lines through $y$ and one of them is in $S^\ast$. By Lemma \ref{lem4.21}, the line belonging to $S^\ast$ cannot be contained in $\mH$. Therefore the five lines of $\mH$ through $y$, going back to $x$ are uniquely determined by $x$ and $y$. 

Now let $y' \in \mO_{3b}$ be a point on one of these five lines and $Q$ the unique quad through $yy'$ and $L_{y'} \not= yy'$. By Lemma \ref{lem4.10}, $L_{y'}$ meets $\mathcal{O}_4$. By Lemma~\ref{lem4.22} the third line of $Q$ through $y'$, call it $M_{y'}$, doesn't lie in $\mH$. We claim that $M_{y'}$ intersects $\mO_{2b}$. Indeed, as the point $x$ is classical with respect to $Q$, the unique point $u$ in $Q$ nearest to $x$ lies at distance 2 from $x$ and is collinear with $y'$. Therefore, $u  \in \mathcal{O}_{2b}$ and $M_{y'} = y'u$. The four lines of $\mH$ through $y'$ that go back to $x$ are now uniquely determined. Indeed, by Lemma \ref{lem4.24}, each of the four lines of $\mH$ through $y'$ meets $\mathcal{O}_{2b}$. But by the suborbit diagram, there are precisely five such lines. Moreover, one of these five lines is the line $M_{y'}$ and we already know that it cannot be a line of $\mH$.

Now, let $y'' \in \mO_{2b}$ be a point on one of these four lines. By the suborbit diagram there is a unique line through $y''$ containing a point $y'''$ in $\mO_{1b}$, which must necessarily be in $\mH$. Moreover, there is a unique line through $y'''$ that contains $x$. 

So far, we have proved that given any point $y$ in $\mH$ with $\dist_\mH(x,y) = 4$, all shortest paths between $x$ and $y$ in $\mH$ are uniquely determined by $x$ and $y$. Moreover, all points at distance $4$ from $x$ that are collinear with $y$ are uniquely determined. These properties in fact imply that the whole of $\mathcal{H}$ is uniquely determined. Indeed, the subgraph of the collinearity graph induced on the set $\Gamma_4(x) \cap \mathcal{H}$ is connected (see Step 1 of the proof of Theorem 3 in \cite{Co-Ti}), and every shortest path between $x$ and a point of $\mathcal{H}$ can be extended to a shortest path between $x$ and a point of $\Gamma_4(x) \cap \mathcal{H}$.
\end{proof}

\begin{lem}\label{lem4.26}
There are precisely $416$ Hall-Janko suboctagons of $\mG$, namely the suboctagons $\mathcal{S}_H$ for maximal subgroups $H \cong \mathrm{J}_2{:}2$ of $G = \mathrm{G}_2(4){:}2$.  Through every pair of opposite points of $\mathcal{G}$, there is precisely one Hall-Janko suboctagon.
\end{lem}
\begin{proof}
A Hall-Janko suboctagon has $315 \cdot 64$ ordered pairs of opposite points while $\mG$ has $4095 \cdot 2048$ such pairs. Therefore by Lemma~\ref{lem4.25}, there are at most $(4095 \cdot 2048)/(315 \cdot 64) = 416$ Hall-Janko suboctagons in $\mG$. By Lemma~\ref{lem4.15} there are at least that many. 
\end{proof}

\medskip \noindent Now that we have classified all Hall-Janko sub near octagons of $\mathcal{G}$, we end this section with proving some extra properties of these Hall-Janko suboctagons. 

\begin{lem} \label{lem4.27}
If $\mH$ is a Hall-Janko suboctagon of $\mG$ and  $x$ a point not contained in $\mH$ then there is a unique point $x'$ in $\mH$ that is collinear with $x$. 
\end{lem}
\begin{proof}
Let $\mH$ be a Hall-Janko suboctagon of $\mG$ and $x$ a point not contained in $\mH$. Say $x$ has two neighbours $y$, $z$ in $\mH$. Then by Lemma~\ref{lem4.23} $\dist_\mH(y,z) = 2$ and hence there is a common neighbour of $y$, $z$ inside $\mH$. This means that there is a quad through $y$, $z$ whose intersection with $\mH$ contains a pair of intersecting lines, contradicting Lemma~\ref{lem4.20}. Therefore, if $x$ has a neighbour in $\mH$ then it is unique. 

Now we can show that $x$ has a neighbour in $\mH$ by a simple counting. There are six lines out of the eleven through each point in $\mH$ that are not contained in $\mH$, giving us a total of $12 \cdot 315$ points of $\mG$ at distance $1$ from $\mH$, as they all must be distinct. Adding this to the number of points in $\mH$ we get $315 \cdot 12 + 315 = 4095$ which is the total number of points in $\mG$. 
\end{proof}

\medskip \noindent For a Hall-Janko suboctagon $\mH$ and a point $x$ of $\mG$ we define the projection of $x$ onto $\mH$, $\pi_\mH(x)$, to be $x$ if $x \in \mH$ and the unique point $x' \in \mH$ collinear with $x$ if $x \notin \mH$.

\begin{lem} \label{lem4.28}
Let $\mH$ be a Hall-Janko suboctagon of $\mG$ and $x$, $y$ be two distinct points not contained in $\mH$ such that $\pi_\mH(x) = \pi_\mH(y)$. Then $\mH \cap \G_4(x) \neq \mH \cap \G_4(y)$.
\end{lem}
\begin{proof}
We consider the following two cases:
\begin{enumerate}
\item The point $x$ is collinear with $y$. 
Let $x' = \pi_\mH(x) = \pi_\mH(y)$ and $z \in \G_4(x') \cap \mH$. 
Since $\{x, y, x'\}$ is a line, by (NP2), either $\dist(z,x) = 3$ and $\dist(z,y) = 4$, or $\dist(z,x) = 4$ and $\dist(z,y) = 3$. 
In either case $z$ belongs to only one of $\mH \cap \G_4(x)$, $\mH \cap \G_4(y)$. 

\item The point $x$ is not collinear with $y$. Consider the suborbit diagram with $\omega$ equal to the common projection of $x$ and $y$. 

Let $z$ be a point in $\mO_{3b} \cap \mH$. There are five lines through $z$ going back to $\mO_{2b}$ and four of them are contained in $\mH$. The one line that is not contained in $\mH$ gives us a unique point $z'$ of $\mO_{2b} \setminus \mH$ collinear with $z$.  This in turn gives us a unique point $u$ in $\mO_{1b} \setminus \mH$ collinear with $\omega$ and $z'$. This point has distance $2$ from $z$ and cannot belong to $\mathcal{H}$ by Lemma \ref{lem4.27}.

Conversely, let $u$ be a point in $\mO_{1b} \setminus \mH$. It has sixteen neighbours in $\mO_{2b}$ none of which is contained in $\mH$ by Lemma~\ref{lem4.27}.  By the suborbit diagram and Lemma~\ref{lem4.18}, the projection of each of these sixteen points in $\mH$ must lie in $\mO_{3b}$. Therefore, the ten points of $\mO_{1b} \setminus \mH$ partition the set $\mO_{3b} \cap \mH$, by the distance $2$ map, into ten disjoint sets of size sixteen.

Without loss of generality, say $x \in \mO_{1b}$. Then the sixteen points of $\mO_{3b} \cap \mathcal{H}$ that are at distance $2$ from $x$ are at distance $4$ from $y$. Indeed, if $z \in \mO_{3b} \cap \mathcal{H}$ lies at distance $2$ from $x$, then through $\omega$, there are precisely five lines containing a point at distance $2$ from $z$. Four of these lines are contained in $\mH$ and the fifth line is $\omega x$. So, $y$ which is still on another line through $\omega$ should lie at distance 4 from $z$.
\end{enumerate}
\end{proof}

\subsection{The automorphism group} \label{sec4.3}

In this subsection, we show that the full automorphism group of $\mathcal{G}$ is isomorphic to $G = \mathrm{G}_2(4) : 2$.

\begin{lem} \label{lem4.29}
The group $\mathrm{G}_2(4){:}2$ acts as a group of automorphism of $\mathcal{G}$, where the action on the point set (i.e. the central involutions of $G$) is given by conjugation.
\end{lem}
\begin{proof}
Each $g \in G$ determines an automorphism of $\mathcal{G}$: if $x$ is a central involution and $g \in G$, then $x^g = g^{-1} x g$ is again a central involution. Since the central involutions generate the group $G' = \mathrm{G}_2(4)$ and $C_G(G')=1$, the action of each $g \in G \setminus \{ e \}$ is faithful.
\end{proof}

\begin{lem} \label{lem4.30}
Every automorphism $\theta$ of $\mathcal{G}$ permutes the elements of $S^\ast$ and hence determines an automorphism of $(S^\ast,\mathcal{Q}^\ast) \cong \mathrm{H}(4)^D$.
\end{lem}
\begin{proof}
The automorphism $\theta$ permutes the quads of $\mathcal{G}$ and hence the lines of $\mathcal{G}$ that can be obtained as intersections of two quads.
\end{proof}

\begin{lem} \label{lem4.31}
Suppose $\theta$ is an automorphism of $\mathcal{G}$ fixing each line of $S^\ast$. Then $\theta$ is the identity.
\end{lem}
\begin{proof}
Let $x$ be an arbitrary point of $\mathcal{G}$, $L = \{ x,y,z \}$ a line through $x$ not belonging to $S^\ast$ and $Q$ the unique quad through $L$. The lines of $S^\ast$ contained in $Q$ determine a spread of the $W(2)$-quad $Q$. Inside $Q$, it is easily seen that $L$ is the unique line of $Q$ meeting $L_x$, $L_y$ and $L_z$. From $L_x^\theta = L_x$, $L_y^\theta = L_y$ and $L_z^\theta = L_z$, it then follows that $x^\theta = x$. 
\end{proof}

\begin{prop} \label{prop4.32}
The full automorphism group of $\mathcal{G}$ is isomorphic to $\mathrm{G}_2(4) : 2$.
\end{prop}
\begin{proof}
By Lemmas \ref{lem4.30} and \ref{lem4.31}, $|Aut(\mathcal{G})| \leq |Aut(\mathrm{H}(4)^D)| = |\mathrm{G}_2(4){:}2|$. Lemma \ref{lem4.29} then implies that $Aut(\mathcal{G}) \cong \mathrm{G}_2(4) : 2$.
\end{proof}

\bigskip \noindent It is possible to give another proof of Proposition \ref{prop4.32} based on the following lemma.

\begin{lem} \label{lem4.33}
Let $H$ be a subgroup of $\mathrm{G}_2(4){:}2$ isomorphic to $\mathrm{J}_2{:}2$. Then every automorphism $\theta$ of $\mathcal{S}_H \cong \HJ$ extends to precisely one automorphism of $\mathcal{G}$.
\end{lem}
\begin{proof}
The action of $\theta$ on the point set of $\mathcal{S}_H$ is given by conjugation by a suitable element of $H \cong \mathrm{J}_2{:}2$. This conjugation also determines an automorphism of $\mathcal{G}$. To show that $\theta$ extends to at most one automorphism of $\mathcal{G}$, we must show that every automorphism $\varphi$ of $\mathcal{G}$ that fixes each point of $\mathcal{S}_H$ must be trivial. But this is implied by Lemma \ref{lem4.28}.
\end{proof}

\bigskip \noindent Since there are 416 Hall-Janko suboctagons, Lemma \ref{lem4.33} implies that $|Aut(\mathcal{G})| \leq 416 \cdot |Aut(\HJ)| = 416 \cdot |\mathrm{J}_2{:}2| = |\mathrm{G}_2(4) {:}2|$. Lemma \ref{lem4.29} then again implies that $Aut(\mathcal{G}) \cong \mathrm{G}_2(4) : 2$. In fact, this reasoning also gives that the automorphism group is transitive on the Hall-Janko suboctagons, but we already knew this in advance as all maximal subgroups isomorphic to $\mathrm{J}_2{:}2$ are conjugate.

\section{The Suzuki tower} \label{sec5}

Let $\mathcal{S}_0$, $\mathcal{S}_1$, $\mathcal{S}_2$, $\mathcal{S}_3$ be the near polygons and $\Gamma_0$, $\Gamma_1$, $\Gamma_2$, $\Gamma_3$, $\Gamma_4$ the graphs of the Suzuki tower as mentioned in Section \ref{sec1}. Then we know that $\mathcal{S}_0 = \mathrm{H}(2, 1)$, $\mathcal{S}_1 = \mathrm{H}(2)^D$, $\mathcal{S}_2 = \HJ$ and $\mathcal{S}_3 = \mathcal{G}$, where $\mathrm{H}(2, 1)$ is the unique generalized hexagon of order $(2,1)$. We define $\mathcal{S}_{-1}$ to be the partial linear space on nine points and four lines obtained from the $(3 \times 3)$-grid by removing two disjoint lines (and keeping the points incident with these two lines). We define $\mathcal{S}_{-2}'$ to be a line with three points and $\mathcal{S}_{-2}''$ to be a coclique of size 3 (no lines). 

The graphs $\Gamma_i$ with $i=0,1,2,3$ can all be obtained in a similar way from the near polygons $\mathcal{S}_i$ and some of their subgeometries. It can easily be verified that $\Gamma_0$ is isomorphic to the graph whose vertices are the subgeometries of $\mathcal{S}_0$ isomorphic to $\mathcal{S}_{-1}$, where two such subgeometries are adjacent whenever they intersect in a subgeometry isomorphic to $\mathcal{S}_{-2}'$ or $\mathcal{S}_{-2}''$. The graph $\Gamma_i$ with $i=1,2$ is known to be isomorphic to the graph whose vertices are the subgeometries of $\mathcal{S}_i$ isomorphic to $\mathcal{S}_{i-1}$, where two subgeometries are adjacent whenever they intersect in a subgeometry isomorphic to $\mathcal{S}_{i-2}$. We prove an analogous property\footnote{A similar property holds for the complements: The complement of $\Gamma_i$ with $i=0,1,2,3$ is isomorphic to the graph whose vertices are the subgeometries of $\mathcal{S}_i$ isomorphic to $\mathcal{S}_{i-1}$, where two distinct subgeometries are adjacent whenever they intersect in the perp of a point (in any of these two subgeometries).} for the graph $\Gamma_3$.

\begin{lem} \label{lem5.1}
The $\mathrm{G}_2(4)$-graph $\Gamma_3$ is isomorphic to the graph $\Gamma'$ whose vertices are the Hall-Janko suboctagons of $\mathcal{G}$, where two Hall-Janko suboctagons are adjacent whenever they intersect in a subgeometry isomorphic to $\mathrm{H}(2)^D$.
\end{lem}
\begin{proof}
The $\mathrm{G}_2(4)$-graph is the graph whose vertices are the maximal subgroups of $\mathrm{G}_2(4)$ isomorphic to $\mathrm{J}_2$, where two such maximal subgroups are isomorphic if they intersect in a subgroup isomorphic to $G_2(2)' \cong \mathrm{U}_3(3)$.

It is well-known that the Hall-Janko near octagon $\HJ$ has 100 subhexagons isomorphic to $\mathrm{H}(2)^D$, and that these are in bijective correspondence with the 100 maximal subgroups of $\mathrm{J}_2$ isomorphic to $G_2(2)'$. The points of a subhexagon are the central involutions contained in the corresponding maximal subgroup. Moreover, these central involutions generate the maximal subgroup.

For every subgroup $H$ of $\mathrm{G}_2(4)$, denote by $\Sigma_H$ the set of central involutions contained in $H$. If $H \cong \mathrm{J}_2$, then the geometry $\mathcal{S}_H$ induced on the subspace $\Sigma_H$ is isomorphic to $\HJ$. By Lemma \ref{lem4.26}, the map $H \mapsto \mathcal{S}_H$ defines a bijection between the 416 maximal subgroups of $\mathrm{G}_2(4)$ isomorphic to $\mathrm{J}_2$ and the 416 Hall-Janko suboctagons of $\mathcal{G}$. We show that this map defines an isomorphism between the $\mathrm{G}_2(4)$-graph and the graph $\Gamma'$. Take two mutually distinct subgroups $H_1$ and $H_2$ of $\mathrm{G}_2(4)$ isomorphic to $\mathrm{J}_2$. 

If $H_1$ and $H_2$ are two adjacent vertices of the $\mathrm{G}_2(4)$-graph, then the subgeometries $\mathcal{S}_{H_1}$ and $\mathcal{S}_{H_2}$ intersect in a subgeometry whose point set is $\Sigma_{H_1} \cap \Sigma_{H_2} = \Sigma_{H_1 \cap H_2}$, i.e. in a subgeometry isomorphic to $\mathrm{H}(2)^D$ as $H_1 \cap H_2 \cong G_2(2)'$.

Conversely, suppose that $\mathcal{S}_{H_1}$ and $\mathcal{S}_{H_2}$ intersect in a subgeometry isomorphic to $\mathrm{H}(2)^D$. Then $\Sigma_{H_1} \cap \Sigma_{H_2}$ contains all central involutions that are contained in a certain $G_2(2)'$-subgroup $K_i$ of $H_i \cong \mathrm{J}_2$, $i=1,2$. Since all these central involutions generate $K_i$, the groups $K_1$ and $K_2$ are equal, say to $K$. As $K$ is a maximal subgroup of both $H_1$ and $H_2$, we have $K = H_1 \cap H_2$, i.e. $H_1$ and $H_2$ are adjacent in the $\mathrm{G}_2(4)$-graph.
\end{proof}

\bigskip \noindent  Lemma \ref{lem5.1} is precisely Theorem \ref{main:3}. The graphs $\Gamma_0,\Gamma_1,\Gamma_2,\Gamma_3$ of the Suzuki tower can all be constructed from the near polygons $\mathcal{S}_0,\mathcal{S}_1,\mathcal{S}_2,\mathcal{S}_3$ and some of their subgeometries. Theorem \ref{main:4} which we will now prove says that this is also true for the remaining graph $\Gamma_4$ in the Suzuki tower. In fact, the construction given in Theorem \ref{main:4}  is a translation (in terms of subgeometries of $\mathcal{S}_3 = \mathcal{G}$) of the original construction of the Suzuki graph \cite{Su}. We wish to note that it is also possible to give similar constructions for the other graphs of the Suzuki tower by translating their original constructions in terms of substructures of a suitable $\mathcal{S}_i$. We will omit these other constructions here.

\medskip \noindent Let us first review the original construction of the Suzuki Tower (see \cite{Su} or \cite{Pa}). Let $\Delta = \G_{i-1}$, $\G = \G_i$ and $H = Aut(\Delta)$ for some $i \in \{1, 2, 3, 4\}$. The graph $\G$ is constructed from the graph $\Delta$ as follows. 

Let $\infty$ be an extra symbol. Let $S$ be the conjugacy class of $2$-subgroups of type $2A$ if $i = 1$, $2$ or $3$. In the remaining case, $i = 4$, let $S$ be the conjugacy class of $2^2$-subgroups of $H$ generated by $2A$-involutions $x$ and $y$ such that $[H : N_H(\langle x, y \rangle)] = 1365$ (the long root subgroups). Define the vertex set of $\G$ as $V(\G) := \{\infty\} \cup V(\Delta) \cup S$. Vertex $\infty$ is made adjacent to all the elements of $V(\Delta)$, two vertices in $V(\Delta)$ are adjacent if they are adjacent as vertices of $\Delta$, a vertex $x \in S$ is adjacent to a vertex $v \in V(\Delta)$ if a non trivial element of the subgroup corresponding to $x$ fixes $v$, and two vertices $x$, $y$ in $S$ are adjacent if $x$, $y$ considered as subgroups of $H$ do not commute but there exists a $z \in S$ that commutes with both of them. 

\medskip \noindent In the case $i=4$, $\Delta$ is the $\mathrm{G}_2(4)$-graph and $H \cong \mathrm{G}_2(4){:}2$. We know that the vertices of $\Delta$ can be put in 1-1 correspondence with the Hall-Janko suboctagons of $\mathcal{G}$ and that the elements of $S$ (the long root subgroups) can be put in 1-1 correspondence with the lines of the spread $S^\ast$. Two long root subgroups are adjacent whenever they do not commute, but there exists a long root subgroup that commutes with both. In terms of properties of $\mathcal{G}$, this means that the corresponding lines of $S^\ast$ must lie at distance $2$ from each other in the near polygon. By Lemma \ref{lem4.33} we know that every automorphism stabilizing a Hall-Janko suboctagon $\mathcal{H}$ must be a conjugation by an element of the $\mathrm{J}_2{:}2$-subgroup corresponding to $\mathcal{H}$.
From this it follows that if $x \in S$ and $v \in V(\Delta)$, then a non-trivial element of the subgroup corresponding to $x$ fixes $v$ if and only if the spread line corresponding to $x$ intersects the Hall-Janko suboctagon corresponding to $v$. This all implies that the graph as defined in Theorem \ref{main:4} must be isomorphic to the Suzuki graph $\Gamma_4$.

\bigskip \noindent {\bf Remark}: There are other known strongly regular graphs that can be constructed from substructures of these near polygons in a similar way as Theorems \ref{main:3} and \ref{main:4}. For example, it is known that the Hall-Janko near octagon contains $280$ copies of the generalized octagon of order $(2,1)$, denoted as $GO(2,1)$, as convex subgeometries (cf.~Proposition 4.7 in \cite{Yo}), with every pair of distinct $GO(2,1)$'s intersecting in $5$ or $15$ points. Computations showed that the graph defined on these $280$ suboctagons, where adjacency is defined by intersection in $15$ points is an $srg(280,36,8,4)$ (this fact was communicated to Andries Brouwer who has included it on his website \cite{Br}). 

Computations with subgeometries of $\mathcal{G}$ also showed that if $\mN_1$ is a Hall-Janko suboctagon and $\mN_2$ a $\mG'$-suboctagon then $\mN_1 \cap \mN_2$ is either isomorphic to $\mathrm{H}(2, 1)$ or $GO(2,1)$. These intersections give us $56$ suboctagons of $\mG'$ isomorphic to $GO(2,1)$, with every pair of distinct $GO(2,1)$'s intersecting in $5$ or $9$ points. We can define a graph on these $56$ suboctagons by defining adjacency as intersection in $9$ points. Computations revealed that this is an $srg(56, 10, 0, 2)$ necessarily isomorphic to the unique strongly regular graph with those parameters, the well known Sims-Gewirtz graph. 
We can also construct the graph $srg(162, 56, 10, 24)$, which is the second subconstituent of the McLaughlin graph \cite{S93}, by taking the elements of $\{\infty\}$, $A$ and $B$ as vertices where $A$ is the set of $56$ sub $GO(2,1)$'s of $\mG'$ and $B$ the set of $105$ lines of $S^*$ that are contained in $\mG'$. Then join $\infty$ to all vertices in $A$, join two distinct vertices of $A$ if the corresponding suboctagons intersect in $9$ points, join a vertex of $A$ to all vertices of $B$ that correspond to a line intersecting the suboctagon, and join two vertices of $B$ if the  lines are at distance $2$. 

\appendix \section{The original construction of the near octagon} \label{appA}

In this appendix, we give another description of the near octagon $\mathcal{G}$. This description was the first description we had for this new near octagon and arose while the authors were studying near polygons that contain $\HJ$ as an isometrically embedded subgeometry. We define a {\em valuation} of $\HJ$ as a map from the point set of $\HJ$ to $\N$ satisfying:
\begin{enumerate}
\item[(V1)] There exists at least one point with $f$-value 0.
\item[(V2)] Every line $L$ of $\HJ$ contains a unique point $x_L$ such that $f(x) = f(x_L) + 1$ for every point $x \not= x_L$ of $L$.
\end{enumerate}
Two valuations $f_1$ and $f_2$ of $\HJ$ are called {\em neighboring} if there exists an $\epsilon \in \mathbb{Z}$ (necessarily belonging to $\{ -1,0,1 \}$) such that $|f_1(x)-f_2(x)+\epsilon| \leq 1$ for every point $x$ of $\HJ$. The number $\epsilon$ is uniquely determined, except when $f_1=f_2$, in which case there are three possible values for $\epsilon$, namely $-1$, $0$ and $1$.

Now, suppose that $f_1$ and $f_2$ are two neighboring valuations of $\HJ$ and let $\epsilon \in \{ -1,0,1 \}$ such that $|f_1(x) - f_2(x) + \epsilon| \leq 1$ for every point $x$ of $\HJ$. If $x$ is point such that $f_1(x) = f_2(x)-\epsilon$, then we define $f_3'(x) := f_1(x)-1 = f_2(x)-\epsilon-1$. If $x$ is a point such that $f_1(x) \not= f_2(x)-\epsilon$, then we define $f_3'(x) := \max(f_1(x),f_2(x)-\epsilon)$. If we put $f_3(x) := f_3'(x)-m$, where $m \in \{ -1,0,1 \}$ is the minimal value attained by $f_3'$, then $f_3$ is again a valuation of $\HJ$, which we will also denote by $f_1 \ast f_2$. The map $f_1 \ast f_2$ is well-defined: if $f_1=f_2$, then there are three possibilities for $\epsilon$, but for each of them, we would have $f_1 \ast f_2 = f_1 = f_2$.  The following properties hold: (i) $f_2 \ast f_1 = f_1 \ast f_2 = f_3$; (ii) $f_1$ and $f_3$ are neighboring valuations and $f_1 \ast f_3 = f_2$; (iii) $f_2$ and $f_3$ are neighboring valuations and $f_2 \ast f_3 = f_1$.

The {\em valuation geometry} $\mathcal{V}$ of $\HJ$ is defined as the partial linear space whose points are the valuations of $\HJ$ and whose lines are the triples $\{ f_1,f_2,f_3 \}$, where $f_1$, $f_2$ and $f_3$ are three mutually distinct valuations of $\HJ$ such that $f_1$ and $f_2$ are neighboring valuations and $f_3 = f_1 \ast f_2$. The valuation geometry $\mathcal{V}$ can provide useful information about near polygons containing $\HJ$ as an isometrically embedded full subgeometry:

\begin{lem} \label{lem6.1}
Let $\mathcal{S}$ be a near polygon with three points per line containing $\HJ$ as an isometrically embedded subgeometry. For every point $x$ of $\mathcal{S}$, let $f_x$ be the map from the point set of $\HJ$ to $\N$ sending each point $y$ of $\HJ$ to $\dist(x,y)-m$, where $m$ is the distance between $x$ and $\HJ$. Then the following holds:
\begin{enumerate}
\item[$(1)$] For every point $x$ of $\mathcal{S}$, $f_x$ is a valuation of $\HJ$.
\item[$(2)$] For every line $\{ x,y,z \}$ of $\mathcal{S}$, either $f_x=f_y=f_z$ or $\{ f_x,f_y,f_z \}$ is a line of $\mathcal{V}$.
\end{enumerate}
\end{lem}
\begin{proof}
This is easily deduced from Property (NP2) in the definition of valuation, see e.g. \cite[Lemma 2.2]{ab-bdb:1}.
\end{proof}

\bigskip \noindent For the purpose of studying near polygons containing $\HJ$ as an isometrically embedded subgeometry, the authors determined all valuations of $\HJ$ with the aid of GAP, see \cite{ab-bdb:2}. It turns out that there are 7119 valuations which fall into five isomorphism classes, see Table \ref{tab1}. In this table, $M_f$ denotes the maximal value attained by a valuation $f$ and $\mathcal{O}_f$ denotes the set of points with value 0. The number of points with a given value $i \in \{ 0,1,2,3,4 \}$ can be found as the $(i+1)$-th entry in ``value distribution''.

\begin{table}
\begin{center}
\begin{tabular}{|c||c|c|c|c|c|}
\hline
Type & \# & $M_f$ & $|\mathcal{O}_f|$ & value distribution \\
\hline
\hline
$A$ & 315 & 4 & 1 & [1,10,80,160,64] \\
\hline 
$B$ & 630 & 3 & 1 & [1,10,112,192,0] \\
\hline 
$C$ & 3150 & 3 & 1 & [1,26,128,160,0] \\
\hline 
$D$ & 1008 & 2 & 5 & [5,110,200,0,0] \\
\hline 
$E$ & 2016 & 2 & 25 & [25,130,160,0,0] \\
\hline 
\end{tabular}
\end{center}
\caption{The valuations of Hall-Janko near octagon}
\label{tab1}
\end{table}

Subsequently, we have determined the possible line types for the lines of $\mathcal{V}$, together with information saying how many lines of each type are incident with a given point of Type $T \in \{ A,B,C,D,E \}$. This information can be found in Table \ref{tab2}. 

\begin{table}
\begin{center}
\begin{tabular}{|c||c|c|c|c|c|}
\hline
Type & $A$ & $B$ & $C$ & $D$ & $E$ \\
\hline
\hline
$AAA$ & 5 & -- & -- & -- & -- \\
\hline
$ABB$ & 1 & 1 & -- & -- & -- \\
\hline
$ACC$ & 5 & -- & 1 & -- & -- \\
\hline
$BBB$ & -- & 5 & -- & -- & -- \\
\hline
$BBC$ & -- & 10 & 1 & -- & -- \\
\hline
$CCC$ & -- & -- & 9 & -- & -- \\
\hline
$CDD$ & -- & -- & 4 & 25 & -- \\
\hline
$DDD$ & -- & -- & -- & 6 & -- \\
\hline
$DEE$ & -- & -- & -- & 1 & 1 \\
\hline
$EEE$ & -- & -- & -- & -- & 6 \\
\hline 
\end{tabular}
\end{center}
\caption{The lines of the valuation geometry}
\label{tab2}
\end{table}

Now, take the subgeometry $\mathcal{V}'$ of order $(2,10)$ of $\mathcal{V}$ whose points are the valuations of Type $A$, $B$, $C$, and whose lines are the lines of Type $AAA$, $ABB$, $ACC$, $BBC$, $CCC$. Computer computations showed that this is a near octagon (containing $\HJ$ as a full suboctagon). Computer computations also revealed that $\mathrm{G}_2(4){:}2$ was the most likely candidate for the full automorphism group (see Section \ref{sec1}). An attempt to construct the near octagon directly from the group $\mathrm{G}_2(4){:}2$ was successful and resulted in the description given in Theorem \ref{main:1}. We end this appendix by showing that the geometries $\mathcal{G}$ and $\mathcal{V}'$ are indeed isomorphic.

\begin{prop} \label{prop7.2}
The near octagon $\mathcal{G}$ is isomorphic to $\mathcal{V}'$.
\end{prop}
\begin{proof}
Regard $\HJ$ as a full subgeometry of $\mathcal{G}$. Then $\HJ$ is isometrically embedded into $\mathcal{G}$ by Lemma \ref{lem4.23}. By Lemma \ref{lem6.1}, every point $x$ of $\mathcal{G}$ will induce a valuation $f_x$ of $\HJ$. This valuation is of Type A if and only if $x$ belongs to $\HJ$. By Lemma \ref{lem4.27}, each induced valuation has a unique point with value 0. So, all induced valuations have Type A, B or C. By Lemma \ref{lem4.28}, all induced valuations are distinct, implying that the 4095 induced valuations are precisely the 4095 valuations of $\HJ$ that have Type A, B or C. Now, every point of $\mathcal{G}$ is incident with precisely 11 lines. By looking at the columns ``A'' and ``C'' of Table \ref{tab2}, we see that all lines of $\mathcal{V}$ of Type AAA, ABB, ACC, BBC and CCC should be induced (in the sense of Lemma \ref{lem6.1}(2)). The number of such lines of $\mathcal{V}$ is equal to $\frac{315 \cdot 5}{3} + 315 \cdot 1 + 315 \cdot 5 + 3150 \cdot 1 + \frac{3150 \cdot 9}{3} = 15015$. Since $\mathcal{G}$ has $\frac{4095 \cdot 11}{3} = 15015$ lines, we see that the lines of $\mathcal{V}$ that are induced are precisely the lines of Type AAA, ABB, ACC, BBC and CCC. We can now conclude that $\mathcal{G}$ and $\mathcal{V}'$ are isomorphic.
\end{proof}

\subsection*{Acknowledgment}

The authors wish to thank Andries Brouwer for discussions on the topics of the paper and his helpful comments.
They also thank the anonymous referees for their detailed reports and constructive remarks.

\end{document}